\documentclass[12pt]{amsart}
\usepackage{amsmath,amsthm,amsfonts,amssymb,latexsym,enumerate,url,hyperref}

\begin{document}

\theoremstyle{plain}

\newtheorem{thm}{Theorem}[section]
\newtheorem{lem}[thm]{Lemma}
\newtheorem{conj}[thm]{Conjecture}
\newtheorem{pro}[thm]{Proposition}
\newtheorem{cor}[thm]{Corollary}
\newtheorem{que}[thm]{Question}
\newtheorem{rem}[thm]{Remark}
\newtheorem{defi}[thm]{Definition}
\newtheorem{hyp}[thm]{Hypothesis}

\newtheorem*{thmA}{THEOREM A}
\newtheorem*{corB}{COROLLARY B}

\newtheorem*{thmC}{THEOREM C}
\newtheorem*{conjA}{CONJECTURE A}
\newtheorem*{conjB}{CONJECTURE B}
\newtheorem*{conjC}{CONJECTURE C}

\newtheorem*{thmAcl}{Main Theorem$^{*}$}
\newtheorem*{thmBcl}{Theorem B$^{*}$}

\numberwithin{equation}{section}

\newcommand{\Maxn}{\operatorname{Max_{\textbf{N}}}}
\newcommand{\Syl}[2]{{\rm Syl}_{#1}({#2})}
\newcommand{\dl}{\operatorname{\mathfrak{d}}}
\newcommand{\Con}{\operatorname{Con}}
\newcommand{\cl}{\operatorname{cl}}
\newcommand{\Stab}{\operatorname{Stab}}
\newcommand{\Aut}{\operatorname{Aut}}
\newcommand{\Ker}{\operatorname{Ker}}
\newcommand{\IBr}{\operatorname{IBr}}
\newcommand{\card}[1]{|#1|}
\newcommand{\Irr}[1]{{\rm Irr}({#1})}
\newcommand{\SL}{\operatorname{SL}}
\newcommand{\FF}{\mathbb{F}}
\newcommand{\NN}{\mathbb{N}}
\newcommand{\C}[2]{\textbf{C}_{#1}({#2})}
\newcommand{\N}[2]{\textbf{N}_{#1}({#2})}
\newcommand{\OO}{\mathbf{O}}
\newcommand{\F}{\mathbf{F}}
\newcommand{\aplii}[4]{\begin{array}{ccc}
	#1 &\longrightarrow & #2\\
	#3 &\longmapsto & #4
	\end{array}}

\renewcommand{\labelenumi}{\upshape (\roman{enumi})}

\newcommand{\GL}{\operatorname{GL}}
\newcommand{\Sp}{\operatorname{Sp}}
\newcommand{\PGL}{\operatorname{PGL}}
\newcommand{\PSL}{\operatorname{PSL}}
\newcommand{\SU}{\operatorname{SU}}
\newcommand{\PSU}{\operatorname{PSU}}
\newcommand{\PSp}{\operatorname{PSp}}

\newcommand{\Hall}[2]{{\rm Hall}_{#1}({#2})}

\providecommand{\V}{\mathrm{V}}
\providecommand{\E}{\mathrm{E}}
\providecommand{\ir}{\mathrm{Irr_{rv}}}
\providecommand{\Irrr}{\mathrm{Irr_{rv}}}
\providecommand{\re}{\mathrm{Re}}

\def\irrp#1{{\rm Irr}_{p'}(#1)}
\def\Irrs#1{{\rm Irr}_{\sigma'}(#1)}
\def\ibrp#1{{\rm IBr}_{p'}(#1)}
\def\lin#1{{\rm Lin}(#1)}

\newcommand{\Z}[1]{\textbf{Z}({#1})}
\def\Q{{\mathbb Q}}
\def\irr#1{{\rm Irr}(#1)}
\def\ibr#1{{\rm IBr}(#1)}
\def\irra#1{{\rm Irr}_{\rm A}(#1)}
\def\ibra#1{{\rm IBr}_{\rm A}(#1)}
\def \c#1{{\cal #1}}
\def\cent#1#2{{\bf C}_{#1}(#2)}
\def\nor{\trianglelefteq\,}
\def\oh#1#2{{\bf O}_{#1}(#2)}
\def\Oh#1#2{{\bf O}^{#1}(#2)}
\def\zent#1{{\bf Z}(#1)}
\def\det#1{{\rm det}(#1)}
\def\ker#1{{\rm ker}(#1)}
\def\norm#1#2{{\bf N}_{#1}(#2)}
\def\alt#1{{\rm Alt}(#1)}
\def\iitem#1{\goodbreak\par\noindent{\bf #1}}
   \def \mod#1{\, {\rm mod} \, #1 \, }
\def\sbs{\subseteq}

\def\gc{{\bf GC}}
\def\ngc{{non-{\bf GC}}}
\def\ngcs{{non-{\bf GC}$^*$}}
\newcommand{\notd}{{\!\not{|}}}
\newcommand{\Out}{{\mathrm {Out}}}
\newcommand{\Mult}{{\mathrm {Mult}}}
\newcommand{\Inn}{{\mathrm {Inn}}}
\newcommand{\IBR}{{\mathrm {IBr}}}
\newcommand{\IBRL}{{\mathrm {IBr}}_{\ell}}
\newcommand{\IBRP}{{\mathrm {IBr}}_{p}}
\newcommand{\ord}{{\mathrm {ord}}}
\def\id{\mathop{\mathrm{ id}}\nolimits}
\renewcommand{\Im}{{\mathrm {Im}}}
\newcommand{\Ind}{{\mathrm {Ind}}}
\newcommand{\diag}{{\mathrm {diag}}}
\newcommand{\soc}{{\mathrm {soc}}}
\newcommand{\End}{{\mathrm {End}}}
\newcommand{\sol}{{\mathrm {sol}}}
\newcommand{\Hom}{{\mathrm {Hom}}}
\newcommand{\Mor}{{\mathrm {Mor}}}
\newcommand{\St}{{\sf {St}}}
\def\rank{\mathop{\mathrm{ rank}}\nolimits}
\newcommand{\Tr}{{\mathrm {Tr}}}
\newcommand{\tr}{{\mathrm {tr}}}
\newcommand{\Gal}{{\it Gal}}
\newcommand{\Spec}{{\mathrm {Spec}}}
\newcommand{\ad}{{\mathrm {ad}}}
\newcommand{\Sym}{{\mathrm {Sym}}}
\newcommand{\Char}{{\mathrm {char}}}
\newcommand{\pr}{{\mathrm {pr}}}
\newcommand{\rad}{{\mathrm {rad}}}
\newcommand{\abel}{{\mathrm {abel}}}
\newcommand{\codim}{{\mathrm {codim}}}
\newcommand{\ind}{{\mathrm {ind}}}
\newcommand{\Res}{{\mathrm {Res}}}
\newcommand{\Ann}{{\mathrm {Ann}}}
\newcommand{\Ext}{{\mathrm {Ext}}}
\newcommand{\Alt}{{\mathrm {Alt}}}
\newcommand{\AAA}{{\sf A}}
\newcommand{\SSS}{{\sf S}}
\newcommand{\CC}{{\mathbb C}}
\newcommand{\CB}{{\mathbf C}}
\newcommand{\RR}{{\mathbb R}}
\newcommand{\QQ}{{\mathbb Q}}
\newcommand{\ZZ}{{\mathbb Z}}
\newcommand{\KK}{{\mathbb K}}
\newcommand{\NB}{{\mathbf N}}
\newcommand{\ZB}{{\mathbf Z}}
\newcommand{\OB}{{\mathbf O}}
\newcommand{\EE}{{\mathbb E}}
\newcommand{\PP}{{\mathbb P}}
\newcommand{\GC}{{\mathcal G}}
\newcommand{\HC}{{\mathcal H}}
\newcommand{\AC}{{\mathcal A}}
\newcommand{\BC}{{\mathcal B}}
\newcommand{\GA}{{\mathfrak G}}
\newcommand{\SC}{{\mathcal S}}
\newcommand{\TC}{{\mathcal T}}
\newcommand{\DC}{{\mathcal D}}
\newcommand{\LC}{{\mathcal L}}
\newcommand{\RC}{{\mathcal R}}
\newcommand{\CL}{{\mathcal C}}
\newcommand{\EC}{{\mathcal E}}
\newcommand{\GCD}{\GC^{*}}
\newcommand{\TCD}{\TC^{*}}
\newcommand{\FD}{F^{*}}
\newcommand{\GD}{G^{*}}
\newcommand{\HD}{H^{*}}
\newcommand{\hG}{\hat{G}}
\newcommand{\hP}{\hat{P}}
\newcommand{\hQ}{\hat{Q}}
\newcommand{\hR}{\hat{R}}
\newcommand{\GCF}{\GC^{F}}
\newcommand{\TCF}{\TC^{F}}
\newcommand{\PCF}{\PC^{F}}
\newcommand{\GCDF}{(\GC^{*})^{F^{*}}}
\newcommand{\RGTT}{R^{\GC}_{\TC}(\theta)}
\newcommand{\RGTA}{R^{\GC}_{\TC}(1)}
\newcommand{\Om}{\Omega}
\newcommand{\eps}{\epsilon}
\newcommand{\varep}{\varepsilon}
\newcommand{\al}{\alpha}
\newcommand{\chis}{\chi_{s}}
\newcommand{\sigmad}{\sigma^{*}}
\newcommand{\PA}{\boldsymbol{\alpha}}
\newcommand{\gam}{\gamma}
\newcommand{\lam}{\lambda}
\newcommand{\la}{\langle}
\newcommand{\ra}{\rangle}
\newcommand{\hs}{\hat{s}}
\newcommand{\htt}{\hat{t}}
\newcommand{\sgn}{\mathsf{sgn}}
\newcommand{\SR}{^*R}
\newcommand{\tn}{\hspace{0.5mm}^{t}\hspace*{-0.2mm}}
\newcommand{\ta}{\hspace{0.5mm}^{2}\hspace*{-0.2mm}}
\newcommand{\tb}{\hspace{0.5mm}^{3}\hspace*{-0.2mm}}
\def\skipa{\vspace{-1.5mm} & \vspace{-1.5mm} & \vspace{-1.5mm}\\}
\newcommand{\tw}[1]{{}^#1\!}
\renewcommand{\mod}{\bmod \,}

\newcommand{\Irre}[1]{{\rm Irr}(#1)}
\newcommand{\Irra}[2]{{\rm Irr}_{#1}(#2)}
\def\st{\text{ }|\text{ }}
\newcommand{\inv}[1]{{#1}^{-1}}

\marginparsep-0.5cm

\renewcommand{\thefootnote}{\fnsymbol{footnote}}
\footnotesep6.5pt

\title{A $\sigma$-McKay theorem for $\pi$-separable groups}
\author{David Cabrera-Berenguer}
\address{Departament de Matem\`atiques, Universitat de Val\`encia, 46100 Burjassot,
Val\`encia, Spain}
\email{david.cabrera@uv.es}

\thanks{This research is supported by Grant PID2022-137612NB-I00
 funded by MCIN/AEI/ 10.13039/501100011033 and ERDF ``A way of making Europe”. This work is part of the author's PhD thesis under the supervision of 
Gabriel Navarro. The author would like to thank Gabriel Navarro, Asier Arranz, and J. Miquel Martínez for helpful 
conversations on the subject.
}



\begin{abstract}
We prove a Hall $\sigma$-subgroup analogue of the McKay conjecture for $\pi$-separable groups proposed by G. 
Navarro. This result simultaneously generalizes the classical McKay conjecture and its $\pi$-separable version. More precisely, if $G$ is $\pi$-separable, $p$ is a prime and $\sigma=\pi\cup\{p\}$, then it is known that there exists $H$ a Hall $\sigma$-subgroup of $G$. In this case, we prove that $|{\rm Irr}_{\sigma'}(G)|=|{\rm Irr}_{\sigma'}(\mathbf{N}_{G}(H))|$.
 \end{abstract}

\maketitle

\section{Introduction} 
The McKay conjecture (now a theorem, see \cite{CS26}) establishes that there exists a bijection
\[f:\Irra{p'}G\to\Irra{p'}{\N GP},\]
where $G$ is a finite group, $p$ is a prime, $P\in\Syl p G$ and $\Irra{p'}G$ is the set of irreducible complex characters of $G$ whose degree is not divisible by $p$.

It is a longstanding general idea in group theory to replace a prime $p$ by a set of primes $\pi$ in various parts of the field. If $G$ is a finite group, our focus here is on the set
\[\Irra{\pi'}G=\bigcap_{p \in \pi}\Irra{p'}G.\]

The first obvious obstacle to doing this is that, in general, Hall $\pi$-subgroups need not exist outside the class of $\pi$-separable groups. The $\pi$-version of the McKay conjecture, however, holds for $\pi$-separable groups by work of T. R. Wolf (see \cite{W78a}). That is, if $\pi$ is a set of primes and $H$ is a Hall $\pi$-subgroup of a $\pi$-separable group $G$, then
\[\card{\Irra{\pi'}G}=\card{\Irra{\pi'}{\N G H}}.\]

Many finite non-$\pi$-separable groups possess Hall $\pi$-subgroups. However, almost as often as such subgroups exist, the $\pi$-version of the McKay conjecture turns out to be false for these groups. For instance, let $G=J_4$ and $\pi=\{5,7\}$. Then $G$ has a unique conjugacy class of abelian Hall $\pi$-subgroups $H$, but the number of irreducible characters of $G$ of $\pi'$-degree does not coincide with the corresponding number for $\N G H$. Hence, the idea that the $\pi$-version of the McKay conjecture might hold for finite groups with a nilpotent (or even cyclic) Hall $\pi$-subgroup $H$ must be discarded.

Can the McKay conjecture and Wolf's theorem be incorporated into a single unified statement? G. Navarro has proposed an answer to this question, and the aim of this paper is to prove it.

\begin{thmA}
Let $G$ be a finite $\pi$-separable group, let $p$ be a prime and let \mbox{$\sigma=\pi\cup\{p\}$}. Then $G$ has a unique conjugacy class of Hall $\sigma$-subgroups $H$, and
    \[\card{\Irra{\sigma'}G}=\card{\Irra{\sigma'}{\N G H}}.\]    
\end{thmA}

The existence of $H$ is a theorem of M. Suzuki (see Theorem (3.13) of Chapter 5 of \cite{Su86}). Notice that if $\pi$ is empty, then we recover the classical McKay conjecture, and if $p$ does not divide $\card G$, then we obtain the $\pi$-version of the McKay conjecture for $\pi$-separable groups.

The proof of Theorem A relies on the deep consequences that have emerged from the proof of the McKay conjecture. In particular, we shall use the results of \cite{R23} to obtain Lemma \eqref{rossi}, which is crucial for the final part of the proof.
\section{Proofs}
Our notation for characters follows \cite{Is06} and \cite{N18}. 
If $G$ is $\pi$-separable, it is well-known that $G$ satisfies the property of dominance of the Hall $\pi$-subgroups. If $p$ is a prime, then by the following theorem we also obtain dominance of the Hall $\sigma$-subgroups of $G$, where $\sigma=\pi\cup\{p\}$.
\begin{thm}\label{su}
    Let $G$ be a $\pi$-separable group, let $p$ be a prime and let $\sigma=\pi\cup\{p\}$. Then $G$ has a Hall $\sigma$-subgroup $H$ and every $\sigma$-subgroup of $G$ is contained in a $G$-conjugate of $H$.
\end{thm}
\begin{proof}
    See Theorem (3.13) of Chapter 5 of \cite{Su86}.
\end{proof}
Using now the dominance of the previous theorem, we easily obtain the following key lemma.
\begin{lem}\label{93nav}
    Let $G$ be $\pi$-separable, let $H\in\Hall\sigma G$ and let $L\unlhd  G$. If $\chi\in\Irrs G$, then $\chi_L$ has an $H$-invariant irreducible constituent and any two of them are $\N G H$-conjugate.
\end{lem}
\begin{proof}
    Let $\theta\in\irr L$ be a constituent of $\chi_L$. As $\chi(1)$ is a $\sigma'$-number then $|G:G_\theta|$ is a $\sigma'$-number and therefore by the previous theorem we have that $G_\theta$ contains some $G$-conjugate of $H$. As $(G_\theta)^x=G_{\theta^x}$ for each $x\in G$, then necessarily $\chi_L$ contains an $H$-invariant constituent. Now, let $\theta,\theta^x\in\irr L$ be $H$-invariant constituents of $\chi_L$. Then $H\subseteq G_\theta$ and $H\subseteq (G_\theta)^x$. Therefore $H$ and $H^{\inv x}$ are Hall $\sigma$-groups of the $\pi$-separable group $G_\theta$, and by the previous theorem $H^{\inv xc}=H$ for some $c\in G_\theta$. Then $\inv xc\in\N G H$ and $(\theta^x)^{\inv xc}=\theta$, as desired.
\end{proof}
The proof of the $p$-solvable case of the McKay conjecture relies on an argument due to T. Okuyama and M. Wajima. Next we establish the $\pi$-version of this argument. The key step in the proof is an application of the following theorem of T. Wolf.
\begin{thm}\label{wolf}
    Let $K, M\unlhd G$ with $K\leq M$ and $1=(\card K, \card{M/K})=(\card{M/K},\card{G/M})$. Suppose that $H$ is a Hall $\pi$-subgroup of $G$ such that $M=KH$ and $H\cap K=1$. Let $C=\C K H$. Let $\theta\in\Irr K$ be $H$-invariant and let $\theta'\in\Irr C$ be the Glauberman-Isaacs correspondent of $\theta$ under the action of $H$. If $\card K$ is odd or $G/M$ is abelian, then $\theta$ extends to $G$ if and only if $\theta'$ extends to $\N G H$.
\end{thm}
\begin{proof}
    See Theorem (1.5) of \cite{W90}.
\end{proof}
\begin{lem}\label{IsaExtension}
    Let $K,E\unlhd G$ with $K\subseteq E$. Suppose that $(\card K,\card{E:K})=1$, and let $K\subseteq B\subseteq G$. If $\theta\in\Irr K$ is $G$-invariant then, $\theta$ extends to $EB$ if and only if $\theta$ extends to $B$.
\end{lem}
\begin{proof}
    See Lemma (6.8) of \cite{Is18}.
\end{proof}
\begin{lem}\label{okuyama}(Okuyama-Wajima argument)
    Suppose that $H$ is Hall $\pi$-subgroup of $G$, and let $K\unlhd G$ be a $\pi'$-subgroup such that $KH\unlhd G$. Let $N=\N G H$, $C=\C K H$. Let $\theta\in\Irr K$ be $H$-invariant and let $\theta'\in\Irr C$ be the Glauberman-Isaacs correspondent of $\theta$ under the action of $H$ over $K$. Suppose that $C\leq A\leq N$ is such that $A/C$ is abelian. Then $\theta$ extends to $KA$ if and only if $\theta'$ extends to $A$.
\end{lem}
\begin{proof}
    We may assume that $G=KHA$. In the notation of Theorem \eqref{wolf} let $M=KH$. Then obviously $1=(\card K,\card{M/K})=(\card{M/K},\card{G/M})$. Since
    \[G/M=KHA/KH=A/HK\cap A\cong(A/C)/(HK\cap A/C)\]
    is abelian, then $\theta$ extends to $KHA$ if and only if $\theta'$ extends to $\N{KHA}H$.

    Suppose that $\theta$ extends to $KA$. Since $\theta$ is $H$-invariant then $\theta$ is $G$-invariant. Applying now Lemma \eqref{IsaExtension} with $E=KH$, $B=KA$ we obtain that $\theta$ extends to $KHA=G$, and since $A\subseteq\N{KHA}H$, we have that $\theta'$ extends to $A$. 
    
    Assume now that $\theta'$ extends to $A$. Then it suffices to prove that $\theta'$ extends to $\N{KHA}H$. As $A\subseteq N$, by elementary group theory it follows that
    \[\N{HKA} H=\N{HK} HA=\C KHHA=HA.\]

    As the Glauberman-Isaacs correspondence commutes with the $H$-action, $\theta'$ is $H$-invariant, and since $\theta'$ extends to $A$, it follows that $\theta'$ is $HA$-invariant. Also, both $\C K H$ and $\C KHH$ are normal in $HA$. Applying again Lemma \eqref{IsaExtension} with $E=\C K HH$, $B=A$ we have that $\theta'$ extends to $\C K HHA=HA$, yielding the result.
\end{proof}
The following corollary is the $\sigma$-version for $\pi$-separable groups of Theorem (8.11) of \cite{N18}. The proof follows the same lines as that of the cited theorem, replacing a prime $p$ by $\sigma$. In particular, whenever Theorem (8.6) of \cite{N18} is used, we use Lemma \eqref{okuyama} instead, for the Frattini argument we use Theorem \eqref{su}, and the properties of the Glauberman correspondence are replaced by the corresponding properties of the Glauberman-Isaacs correspondence, which can be found in \cite{W78b}. Thus, the proof of this corollary is omitted.
\begin{cor}\label{811nav}
    Let $G$ be $\pi$-separable and let $p$ be a prime. Let $\sigma=\pi\cup\{p\}$ and let $H$ be a Hall $\sigma$-subgroup of $G$ (which exists by Theorem \eqref{su}). Suppose $K\unlhd G$ is a $\sigma'$-group such that $KH\unlhd G$. Let $W\unlhd G$ be a $\sigma$-group and let $\lambda\in\Irr W$ be $G$-invariant. If $\theta\in\irr K$ is $H$-invariant and $\theta'\in\Irr {\C K H}$ is its Glauberman-Isaacs correspondent of $\theta$ with respect to $H$, then
    \[\card{\Irra{\sigma'}{G|\theta\times\lambda}}=\card{\Irra{\sigma'}{\N G H|\theta'\times\lambda}}.\]
\end{cor}
\begin{lem}\label{extpprime}
    Let $N\unlhd G$ be a $p$-group for some prime $p$. Let $\chi\in\Irra{p'}G$ and let $\theta\in\Irr N$ be under $\chi$. Then $\theta$ extends to $G_\theta$.
\end{lem}
\begin{proof}
Working by induction on $\card G$ and considering the Clifford correspondent $\psi\in\Irra{p'}{G_\theta|\theta}$ of $\chi$, we may assume that $\theta$ is $G$-invariant. By Theorem (5.10) of \cite{N18} it suffices to show that $\theta$ extends to $P$ for each Sylow subgroup $P/N$ of $G/N$.
Let $P\in\Syl pG$. Note that $\chi_P$ has some linear constituent, which necessarily lies over $\theta$, and hence $\theta$ extends to $P$. If $Q\in\Syl q G$ and $q\not=p$, since $(o(\theta)\theta(1),\card{QN:N})=1$, we have that $\theta$ extends to $QN$ (see for instance Corollary (6.2) of \cite{N18}).
\end{proof}
\begin{lem}\label{rossi}
    Let $L\unlhd G$ be a $\pi'$-group such that $L/\Z L$ is a direct product of nonabelian simple groups. Let $p$ be a prime, let $Q\in\Syl p L$ and write $\sigma=\pi\cup\{p\}$. Then there is an $\N G Q$-invariant subgroup $\N L Q\subseteq M\subseteq L$, with $M<L$ whenever $Q$ is not normal in $L$, and an $\N G Q$-equivariant bijection $':\Irra{\sigma'}L\to\Irra{\sigma'}{M}$ such that 
    \begin{enumerate}[a)]
        \item $\card{\Irra{\sigma'}{G_\theta|\theta}}=\card{\Irra{\sigma'}{M\N G Q_\theta|\theta'}}$ for each $\theta\in\Irra{\sigma'}L$.
        \item If $\theta\in\Irra{\sigma'}L$ and $\lambda\in\Irr{\Z L}$, then $\theta$ lies over $\lambda$ if and only is so does $\theta'$.
        \item If $M\N G Q=G$, then $M=L$.
    \end{enumerate}
\end{lem}
\begin{proof}
    We follow the notation of Chapter 10 of \cite{N18}. By \cite{CS26} it follows that Corollary (3.4) of \cite{R23} holds for any group. By that corollary and using that $L$ is a $\pi'$-group we obtain an $\N G Q$-invariant subgroup $\N L Q\subseteq M\subseteq L$, with $M<L$ whenever $Q$ is not normal in $L$, and an $\N G Q$-equivariant bijection $':\Irra{p'}L\to\Irra{p'}{M}$ with $(G_\theta,L,\theta)\geq_c(MN_G(Q)_\theta, M,\theta')$ for every $\theta\in\Irra{p'}L$. However, notice that since $L$ is a $\pi'$-group, we have $\Irra{p'}L=\Irra{\sigma'}L$ and $\Irra{p'}M=\Irra{\sigma'}M$. Let $\theta\in\Irra{\sigma'}L$. By Theorem (10.13) of \cite{N18} it follows that the character triples $(G_\theta,L,\theta)$ and $(M\N G Q_\theta,M,\theta')$ are isomorphic. Since character triple isomorphisms preserve character degree ratios and both $\theta(1)$ and $\theta'(1)$ are $\sigma'$-numbers, then necessarily $\card{\Irra{\sigma'}{G_\theta|\theta}}=\card{\Irra{\sigma'}{M\N G Q_\theta|\theta'}}$ and a) follows. Now, by the definition of the central order of the character triples, there are projective representations $\mathcal P$ of $G_\theta$ associated with $\theta$, $\mathcal P'$ of $M\N G Q_{\theta}$ associated with $\theta'$ and a map $\mu:\C{G_\theta}L\to\mathbb C^\times$ such that $\mathcal P_{\C{G_\theta}L}=\mu I_{\theta(1)}$ and $\mathcal P'_{\C{G_\theta}L}=\mu I_{\theta'(1)}$. Hence b) follows by restricting to $\Z L$ and taking traces. For the last part, notice that $M\N G Q_\theta=(M\N G Q)_\theta$. Therefore, if $M\N G Q=G$, then by the definition of character triple isomorphism it follows that $G_\theta/L\cong G_\theta/M$, yielding $M=L$.
\end{proof}
Now we are ready to prove the main result of this paper. In order to do so we establish a relative version with respect to a normal subgroup $Z$, using the techniques of Theorem (10.26) of \cite{N18}. We obtain the desired result by letting $Z=1$.
\begin{thm}
    Let $G$ be $\pi$-separable and let $p$ be a prime. Write $\sigma=\pi\cup\{p\}$ and let $H$ be a Hall $\sigma$-subgroup of $G$ (which exists by Lemma \eqref{su}). Let $Z\unlhd G$ and let $\lambda\in\Irra{\sigma'}Z$ be $H$-invariant. If $\Irra{\sigma'}{G|\lambda}$ is nonempty, then
    \[\card{\Irra{\sigma'}{G|\lambda}}=\card{\Irra{\sigma'}{\N G HZ|\lambda}}.\]
\end{thm}
\begin{proof}
    We proceed by induction on $\card{G:Z}$. Let $N=\N G H$. As $H\subseteq G_\lambda$ then by the Clifford correspondence we have $\card{\Irra{\sigma'}{G|\lambda}}=\card{\Irra{\sigma'}{G_\lambda|\lambda}}$, and since $(NZ)_\lambda=\N{G_\lambda}HZ$ then by the Clifford correspondence we also have $\card{\Irra{\sigma'}{\N G HZ|\lambda}}=\card{\Irra{\sigma'}{\N{G_\lambda}HZ|\lambda}}$. Therefore we may assume that $\lambda$ is $G$-invariant. As a consequence there is a character triple $(G^*,Z^*,\lambda^*)$ isomorphic to $(G,Z,\lambda)$ with $Z^*\subseteq\textbf{Z}(G^*)$ (see for instance Corollary (5.9) of \cite{N18}). If $Z\leq H\leq G$ then we denote $H^*$ as the only subgroup $Z^*\leq H^*\leq G^*$ such that $(H/Z)^*=H^*/Z^*$, where $^*:G/Z\to G^*/Z^*$ is the isomorphism associated to the character triple isomorphism. Notice that $(HZ)^*/Z^*$ is a Hall $\sigma$-subgroup of $G^*/Z^*$. Now we see that there is a unique Hall $\sigma$-subgroup $H^*$ of $G^*$ such that $(HZ)^*=H^*Z^*$. If there are two such subgroups, say $A, B$, then they are Hall $\sigma$-subgroups of $AZ^*$. By Theorem \eqref{su} we have that $A$ is $AZ^*$-conjugate to $B$, and since $Z^*$ is central then necessarily $A=B$. Therefore, $\N{G^*}{H^*}=\N{G^*}{H^*Z^*}$. Using again Theorem \eqref{su} it easily follows that $\N{G/Z}{HZ/Z}=NZ/Z$, and hence
    \[(NZ)^*/Z^*=(NZ/Z)^*={(\N{G/Z}{HZ/Z})^*=\N{G^*/Z^*}{H^*Z^*/Z^*}=\N{G^*}{H^*}/Z^*}.\]
    By the definition of character triple isomorphisms, and using that $\lambda(1)$ is a $\sigma'$-number we deduce that $\card{\Irra{\sigma'}{G|\lambda}}=\card{\Irra{\sigma'}{G^*|\lambda^*}}$ and $\card{\Irra{\sigma'}{NZ|\lambda}}=\card{\Irra{\sigma'}{\N{G^*}{H^*}|\lambda^*}}$. As a consequence, we may assume that $Z\subseteq\textbf{Z}(G)$. We may also assume that $Z$ is a $\sigma'$-group, as we prove next. Write $\sigma=\{p_1,\ldots,p_n\}$ for distinct primes $p_i$ and write $\lambda=\lambda_\sigma\lambda_{\sigma'}$ for its $\sigma$-part and its $\sigma'$-part. Also, write $\lambda_\sigma=\prod_{i=1}^n\lambda_{p_i}$, where each $\lambda_{p_i}\in\Irr Z$ has $p_i$-order. Fix $1\leq i\leq n$ and let $P\in\Syl{p_i}{G}$. View $\lambda_{p_i}\in\Irr{Z_i}$, where $Z_i\in\Syl{p_i}Z$. As $\card{\Irra{\sigma'}{G|\lambda}}\not=0$, we have that $\card{\Irra{{p_i}'}{G|\lambda_{p_i}}}\not=0$, and hence by Lemma \eqref{extpprime} we have that $\lambda_{p_i}\in\Irr{Z_i}$ extends to some $\alpha_i\in\Irr G$. Taking the $p_i$-part of that extension we may assume that $\alpha_i$ has $p_i$-order, and hence by the unicity of the canonical extension (see for instance Corollary (6.2) of \cite{N18}) we obtain that $\lambda_{p_i}\in\Irr Z$ extends to $G$. Hence $\lambda_\sigma$ extends to some $\alpha\in\Irr G$. 
    
    Write $\beta=\inv\alpha$, and for each $Z\leq U\leq G$ consider the map $^*:\Irr{U|\lambda}\to\Irr{U|\lambda_{\sigma'}}$ defined via $\chi\mapsto\beta_U\chi$. Notice that this map with the trivial isomorphism $^*:G/Z\to G/Z$ define a character triple isomorphism between $(G,Z,\lambda)$ and $(G,Z,\lambda_{\sigma'})$. Now, since $\lambda_{\sigma'}$ is $G$-invariant then $\ker{\lambda_{\sigma'}}\unlhd G$ and hence the character triple $(G,Z,\lambda_{\sigma'})$ is isomorphic to $(G/\ker{\lambda_{\sigma'}},Z/\ker{\lambda_{\sigma'}},\lambda_{\sigma'})$. As $\card{Z/\ker{\lambda_{\sigma'}}}$ is the order of $\lambda_{\sigma'}$, by the previous reasoning we may also assume that $Z$ is a $\sigma'$-group.

    Now, let $L/Z$ be a chief factor of $G$. If $\mathcal A$ is a complete set of representatives of the orbits of the action of $N$ on the $H$-invariant characters of $\Irra{\sigma'}{L|\lambda}$, then by Lemma \eqref{93nav} we have $\Irra{\sigma'}{G|\lambda}=\bigcup_{\theta\in\mathcal A}\Irra{\sigma'}{G|\theta}$ and $\Irra{\sigma'}{LN|\lambda}=\bigcup_{\theta\in\mathcal A}\Irra{\sigma'}{LN|\theta}$ are disjoint unions. Hence, as $\card{G:L}<\card{G:Z}$, then by induction it follows that
    $\card{\Irra{\sigma'}{G|\lambda}}=\card{\Irra{\sigma'}{LN|\lambda}}$.
    If $LN<G$ then by induction the result follows, and therefore we assume that $LN=G$. Therefore $LH\unlhd G$.

    If $L/Z$ is a $\pi$-group or a $p$-group then $LH/Z$ is a normal Hall $\sigma$-subgroup of $G/Z$. Thus $LH=ZH$, and since $H\unlhd ZH$ then $H\unlhd G$. Hence we may assume that $L/Z$ is a $\pi'$-group which is not a $p$-group.

    Suppose that $L/Z$ is a $\sigma'$-group. As $Z$ is a $\sigma'$-group then so is $L$. Denote by $':\Irra H{L}\to\Irr{\C L H}$ the Glauberman-Isaacs correspondence. For each $\chi\in\Irra{H}L$ it holds that $\chi'$ is a constituent of $\chi_{\C L H}$ (for $H$ solvable, see Theorem (2.9) of \cite{N18}, and for $\card L$ odd, see Theorem (8.19) of \cite{Is18}), and hence the map $':\Irra H{L|\lambda}\to\Irr{\C L H|\lambda}$ is also a bijection. Now, let $\mathcal B$ be a complete set of representatives of the orbits of the action of $N$ on the $H$-invariant characters of $\Irr{L|\lambda}$. Since the Glauberman-Isaacs correspondence commutes with the $N$-action then $\mathcal B'=\{\theta':\theta\in\Irra{H}{L}\}$ is a complete set of representatives of the orbits of the action of $N$ on $\Irr{\C H L|\lambda}$. Now, combining Lemma \eqref{93nav} and Corollary \eqref{811nav} (applied with $W=1$) we obtain
    \[\card{\Irra{\sigma'}{G|\lambda}}=\sum_{\theta\in\mathcal B}\card{\Irra{\sigma'}{G|\theta}}=\sum_{\theta\in\mathcal B}\card{\Irra{\sigma'}{N|\theta'}}=\card{\Irra{\sigma'}{N|\lambda}}.\]

    Hence we may assume that $L/Z$ is a nonabelian $\pi'$-group and that $p$ divides $\card{L/Z}$. Since $L/Z$ is a nonabelian chief factor and $Z$ is central, then $Z=\Z L$. Let $Q=H\cap L\in\Syl p L$ and notice that $Q$ can't be normal in $L$ since $L/\Z L$ is not a $p$-group. Hence, by Lemma \eqref{rossi} there is an $\N G Q$-invariant subgroup $\N L Q\subseteq M\subseteq L$, with $M<L$, and an $\N G Q$-equivariant bijection $':\Irra{\sigma'}L\to\Irra{\sigma'}{M}$ satisfying the properties of the lemma.

    Let $\mathcal C$ be a complete
set of representatives of the orbits of the action of $N$ on the $H$-invariant characters of $\Irra{\sigma'}{L|\lambda}$. Since the bijection is $\N G Q$-equivariant and $H\subseteq\N G Q$, using part b) of Lemma \eqref{rossi} it follows that $\mathcal C'=\{\theta':\theta\in\mathcal C\}$ a complete
set of representatives of the orbits of the action of $N$ on the $H$-invariant characters of $\Irra{\sigma'}{M|\lambda}$. Since $M\N G Q_\theta=(M\N G Q)_{\theta'}$, using now Lemma \eqref{93nav}, the Clifford correspondence and part a) of Lemma \eqref{rossi} it follows that
\[\begin{aligned}
    \card{\Irra{\sigma'}{G|\lambda}}&=\sum_{\theta\in\mathcal C}\card{\Irra{\sigma'}{G|\theta}}=\sum_{\theta\in\mathcal C}\card{\Irra{\sigma'}{G_\theta|\theta}}\\
    &=\sum_{\theta\in\mathcal C}\card{\Irra{\sigma'}{(M\N G Q)_{\theta'}|\theta'}}=\sum_{\theta\in\mathcal C}\card{\Irra{\sigma'}{M\N G Q|\theta'}}\\
    &=\card{\Irra{\sigma'}{M\N G Q|\lambda}}.
\end{aligned}\]
As $\N G H\subseteq M\N GQ $, then we are done by induction if $M\N GQ<G$. However, if $M\N G Q=G$, then by part c) of Lemma \eqref{rossi} it follows that $L=M$, which is a contradiction.
\end{proof}

\end{document}